\documentclass[amssymb,12pt ]{article}
\pdfoutput=1
\usepackage{geometry}
\usepackage{amsfonts}
\usepackage{amssymb}
\usepackage{amsmath}
\usepackage{amsthm}
\usepackage{graphicx}
\usepackage{picture}
\usepackage{curves}
\usepackage{subfigure}
\usepackage{epic}
\usepackage{here}
\usepackage{color}

\newtheorem{thm}{Theorem}[section]
\newtheorem{cor}[thm]{Corollary}
\newtheorem{lem}[thm]{Lemma}
\newtheorem{pro}[thm]{Proposition}
\newtheorem{defi}[thm]{Definition}

\newenvironment{ack}{\noindent{\bf Acknowledgments}}

\newcommand{\vol}{{\rm vol}}
\newcommand{\cs}{{\rm cs}}
\newcommand{\li}{{\rm Li}_2}

\newcommand{\modulo}{~~({\rm mod}~\pi^2)}

\begin{document}

\title{Conneted sum of representations of knot groups
}
\author{\sc Jinseok Cho}
\maketitle
\begin{abstract}
When two boundary-parabolic representations of knot groups are given, 
we introduce the connected sum of these representations
and show several natural properties including the unique factorization property.
Furthermore, the complex volume of the connected sum is the sum of each complex volumes modulo $i\pi^2$
and the twisted Alexander polynomial of the connected sum is
the product of each polynomials with normalization.

\end{abstract}

\section{Introduction}\label{sec1}

For any oriented knots $K_1$ and $K_2$, the connected sum $K_1\# K_2$ is well-defined and has many natural properties.
For example, any knot can be uniquely decomposed into prime knots. Also, the simplicial volumes 
$\vol(K_1)$, $\vol(K_2)$ and $\vol(K_1\#K_2)$ of $K_1$, $K_2$ and $K_1\#K_2$, respectively, satisfy
$\vol(K_1\# K_2)=\vol(K_1)+\vol(K_2)$. 
Furthermore, for the Alexander polynomials $\Delta_{K_1}$, $\Delta_{K_2}$ and $\Delta_{K_1\#K_2}$
of $K_1$, $K_2$ and $K_1\#K_2$, respectively, 
we have $\Delta_{K_1\# K_2}=\Delta_{K_1}\cdot \Delta_{K_2}$.

On the other hand, many important invariants are defined for a boundary-parabolic representation 
$\rho:\pi_1(K)\rightarrow {\rm PSL}(2,\mathbb{C})$
and its lift $\widetilde{\rho}:\pi_1(K)\rightarrow {\rm SL}(2,\mathbb{C})$ of the knot group $\pi_1(K)$,
where {{\it the knot group} is the fundamental group of the knot complement $\mathbb{S}^3\backslash K$ and} 
the {\it boundary-parabolic}\footnote{Boundary-parabolic representation is also called {\it parabolic representation} in many other texts.} means any meridian loop of the boundary-torus maps to a parabolic element in ${\rm PSL}(2,\mathbb{C})$
under $\rho$.
For example, the complex volume $\vol(\rho)+i\,\cs(\rho)$ and the twisted Alexander polynomial $\Delta_{K,\widetilde{\rho}}$
are some of the important invariants.

For two boundary-parabolic representations $\rho_1:\pi_1(K_1)\rightarrow {\rm PSL}(2,\mathbb{C})$ and
$\rho_2:\pi_1(K_2)\rightarrow {\rm PSL}(2,\mathbb{C})$, we will define the connected sum of $\rho_1$ and $\rho_2$
\begin{equation*}
\rho_1\#\rho_2 :\pi_1(K_1\# K_2)\rightarrow {\rm PSL}(2,\mathbb{C})
\end{equation*} 
in Section \ref{sec2}. 
{\color{red}(Note: After the publication of this article, serious errors were found. 
The author wanted to preserve the content of the publication, so he added the errata in the appendix.)} 
Then this definition satisfies the unique factorization property;
for any oriented knot $K=K_1\#\ldots \# K_g$ and any boundary-parabolic representation $\rho:\pi_1(K)\rightarrow {\rm PSL}(2,\mathbb{C})$,
there exist unique boundary-parabolic representations
\begin{equation*}
\rho_j:\pi_1(K_j)\rightarrow {\rm PSL}(2,\mathbb{C})~(j=1,\ldots,g)
\end{equation*}
satisfying $\rho=\rho_1\#\ldots\#\rho_g$ up to conjugate.
(If two same knots $K_j$ and $K_k$ appear in $K$, then the indices of $\rho_j$ and $\rho_k$ can be exchanged.)

Using this definition, we will show the following additivity of complex volumes
\begin{equation}\label{add}
\vol(\rho_1\#\rho_2)+i\,\cs(\rho_1\#\rho_2)\equiv(\vol(\rho_1)+i\,\cs(\rho_1))+(\vol(\rho_2)+i\,\cs(\rho_2))~~({\rm mod}~i\pi^2),
\end{equation}
in Section \ref{sec3}. The author believes (\ref{add}) was already known to some experts because
the knot complement $\mathbb{S}^3\backslash(K_1\#K_2\cup\{\text{two points}\})$ is obtained by gluing $\mathbb{S}^3\backslash (K_1\cup\{\text{two points}\})$
and $\mathbb{S}^3\backslash (K_2\cup\{\text{two points}\})$ along $\mathbb{T}^2\backslash\{\text{two points}\}$, a torus minus two points.\footnote{The gluing map here is topologically unique because it is obtained by gluing two pairs of two vertex-oriented ideal triangles.
} However, the proof in Section \ref{sec3} will be combinatorial and very simple.
Furthermore, while proving (\ref{add}), we will show the solutions of the hyperbolicity equations $\mathcal{I}_1$ and $\mathcal{I}_2$,
which correspond to the five-term triangulations of $\mathbb{S}^3\backslash (K_1\cup\{\text{two points}\})$ and $\mathbb{S}^3\backslash (K_2\cup\{\text{two points}\})$, respectively,
are determined by the solution of $\mathcal{I}$, which corresponds to the triangulation of $\mathbb{S}^3\backslash(K_1\#K_2\cup\{\text{two points}\})$. (See Lemma \ref{lem3}.)
This is not a usual situation because, in general, if we glue two manifolds, then the set of the hyperbolicity equations changes,
and even small change on the equations induces radical change on the solutions.
Therefore, the solution of the glued manifold usually cannot detect the solutions of the original two manifolds.
However, it works for our case in Section \ref{sec3} because we will use combinatorial method.

In Section \ref{sec4}, we will show
the twisted Alexander polynomial $\Delta_{K_1\#K_2,\widetilde{\rho_1}\#\widetilde{\rho_2}}$ is
the product of $\Delta_{K_1,\widetilde{\rho_1}}$ and $\Delta_{K_2,\widetilde{\rho_2}}$ with normalization.
Finally, Section \ref{sec5} will discuss an example $\rho_1\#\rho_2:\pi_1(3_1\#4_1)\rightarrow{\rm PSL}(2,\mathbb{C})$
and its lift $\widetilde\rho_1\#\widetilde\rho_2:\pi_1(3_1\#4_1)\rightarrow{\rm SL}(2,\mathbb{C})$.

Although we restrict our attention to boundary-parabolic representations for simplicity, 
under certain condition,
all results in Section \ref{sec2} and \ref{sec4} 
are still true for general representations. It will be discussed briefly later.

Note that all representations in this article are defined up to conjugate.
We follow the definition of the complex volume of a representation $\rho$ in \cite{Zickert09}
and that of the twisted Alexander polynomial in Section 2 of \cite{Morifuji08}.

\section{Definition and the unique factorization}\label{sec2}
\subsection{Writinger presentation and arc-coloring}

For a fixed oriented knot diagram\footnote{
We assume $D$ has at least one crossing.} $D$ of a knot $K$, let $\alpha_1,\ldots,\alpha_n$ be the arcs of $D$. 
These arcs can be regarded as the meridian loops of the boundary-torus, which is expressed by small arrows in Figure \ref{pic01}. 
Then Wirtinger presentation gives a presentation of the knot group
\begin{equation}\label{Wpre}
  \pi_1(K)=<\alpha_1,\ldots,\alpha_n \,;\, r_1,\ldots,r_{n-1}>,
\end{equation}
where the relations $r_1,\ldots,r_n$ are defined in Figure \ref{pic02}.
(We can remove one relation in $\{r_1,\ldots,r_n\}$ because it can be obtained by all the others.)

\begin{figure}[h]
\centering
  \includegraphics[scale=0.5]{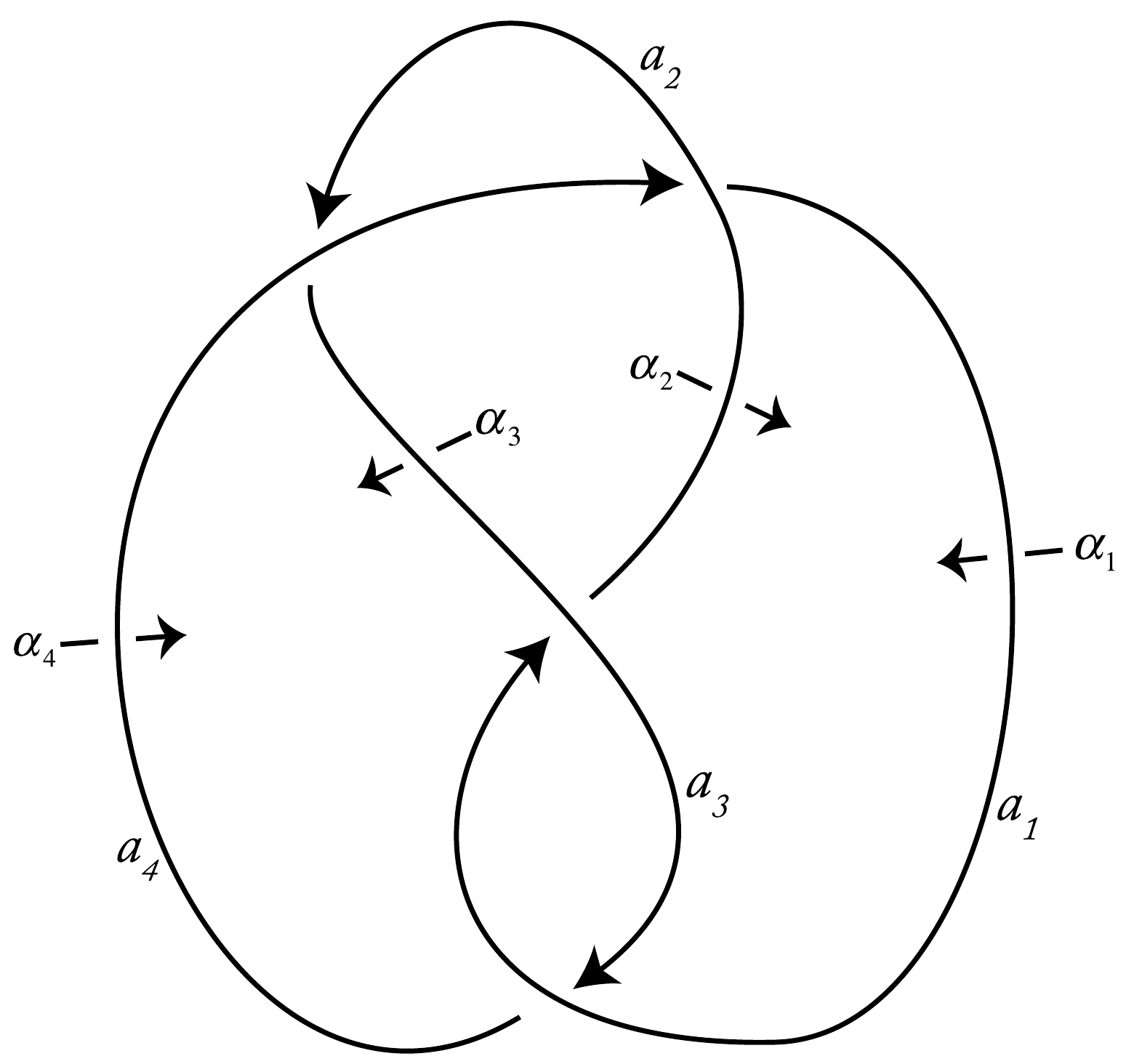}
  \caption{Knot diagram with arcs $\alpha_1,\ldots,\alpha_n$ and arc-colors $a_1,\ldots,a_n$}\label{pic01}
\end{figure}

\begin{figure}[h]
\centering
\subfigure[$r_l : \alpha_{l+1}=\alpha_k \alpha_l \alpha_k^{-1}$]{
  \includegraphics[scale=0.4]{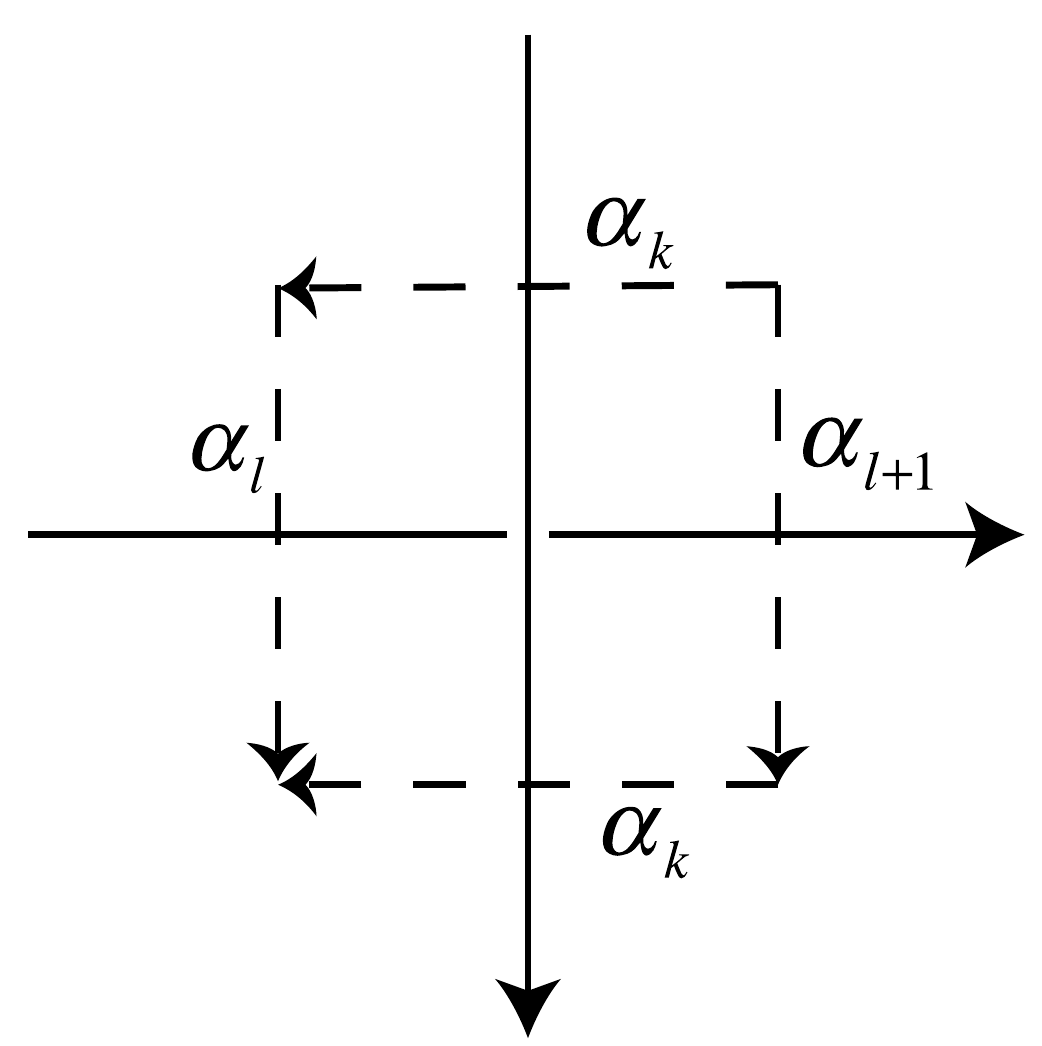}}\hspace{1cm}
  \subfigure[$r_l : \alpha_l=\alpha_k \alpha_{l+1}\alpha_k^{-1}$]{
  \includegraphics[scale=0.4]{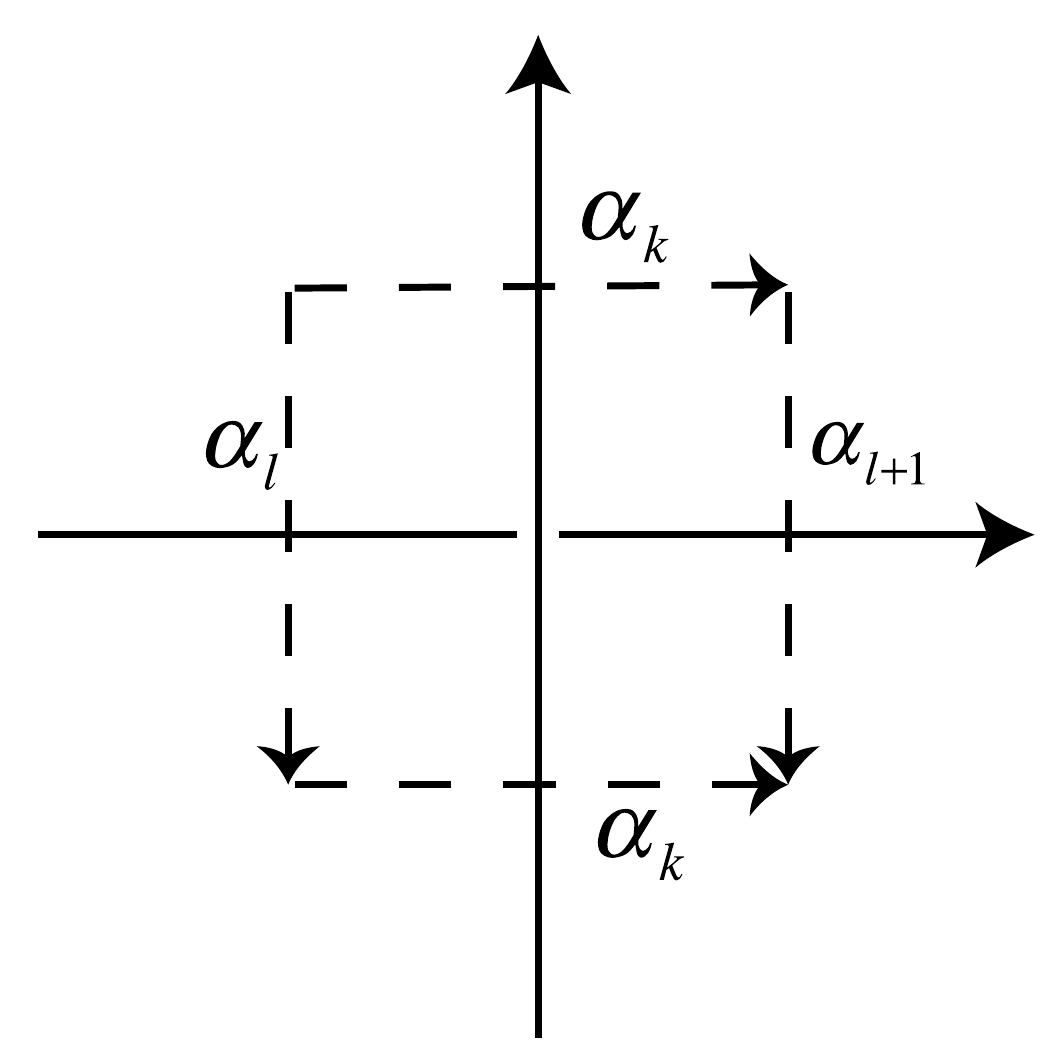} } 
  \caption{Relations at crossings}\label{pic02}
\end{figure}

Let $\mathcal{P}$ be the set of parabolic elements in ${\rm PSL}(2,\mathbb{C})$.
For a boundary-parabolic representation $\rho:\pi_1(K)\rightarrow{\rm PSL}(2,\mathbb{C})$,
put $a_k=\rho(\alpha_k)\in\mathcal{P}$ and call $a_k$ {\it the arc-color of $\alpha_k$} (induced by $\rho$.)
Note that, due to the Wirtinger presentation, the arc-coloring determines the representation $\rho$ uniquely (up to conjugate.) Therefore, from now on, we express the representation $\rho$ by using the arc-coloring of a diagram $D$. 

For $a,b\in\mathcal{P}$, we define the operation $*$ by
\begin{equation}\label{operation}
  a*b=bab^{-1}\in{\rm PSL}(2,\mathbb{C}).
\end{equation}
Then the arc-colors of a crossing satisfy the relation in Figure \ref{pic03}.
Furthermore, the operation {$*b:a\mapsto a*b$} is bijective and satisfies
$$a*a=a\text{ and }(a*b)*c=(a*c)*(b*c),$$
for any $a,b,c\in\mathcal{P}$, which implies $(\mathcal{P},*)$ is a quandle. (See \cite{Kabaya14} or \cite{Cho14a} for details.)
We define the inverse operation $*^{-1}$ by
$$a*^{-1}c=b~\iff~a=b*c.$$

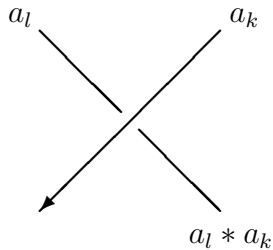
\begin{figure}[h]
\centering  \setlength{\unitlength}{0.6cm}\thicklines
\begin{picture}(6,6)  
    \put(6,5){\vector(-1,-1){4}}
    \put(2,5){\line(1,-1){1.8}}
    \put(4.2,2.8){\line(1,-1){1.8}}
    \put(6.2,5.2){$a_k$}
    \put(1.3,5.2){$a_l$}
    \put(5.3,0.3){$a_l*a_k$}
  \end{picture}
  \caption{Arc-coloring}\label{pic03}
\end{figure}

One trivial, but important fact is that the arc-coloring uniquely changes under the Reidemeister moves.
(This is trivial because arc-coloring is uniquely determined by the representation $\rho$.
Another way to see this fact is to consider the relationship between the Reidemeister moves and the axioms of quandle. See Figure \ref{pic04}.)

\begin{figure}[h]
\centering
\subfigure[$a*a=a$]{
  \includegraphics[scale=1]{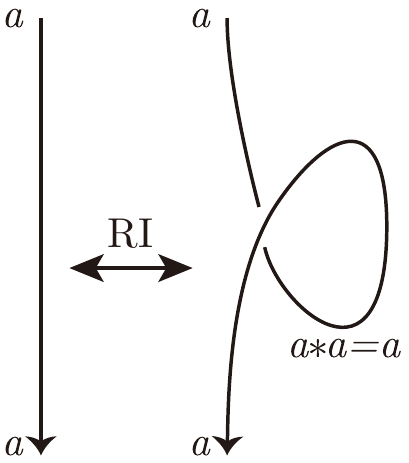}}\hspace{1cm}
  \subfigure[Operation {$*b$} is bijective]{
  \includegraphics[scale=1]{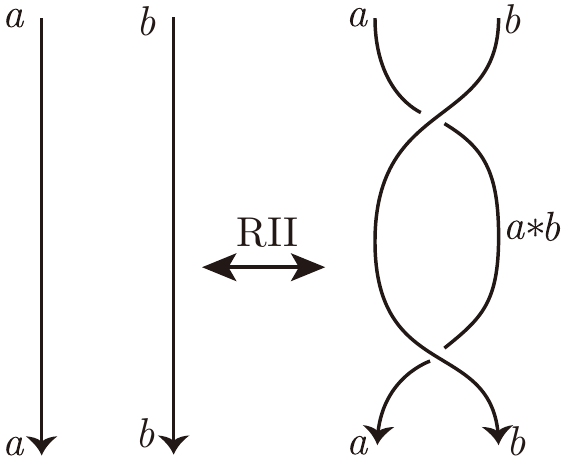} }\\ 
  \subfigure[$(a*b)*c=(a*c)*(b*c)$]{
  \includegraphics[scale=1]{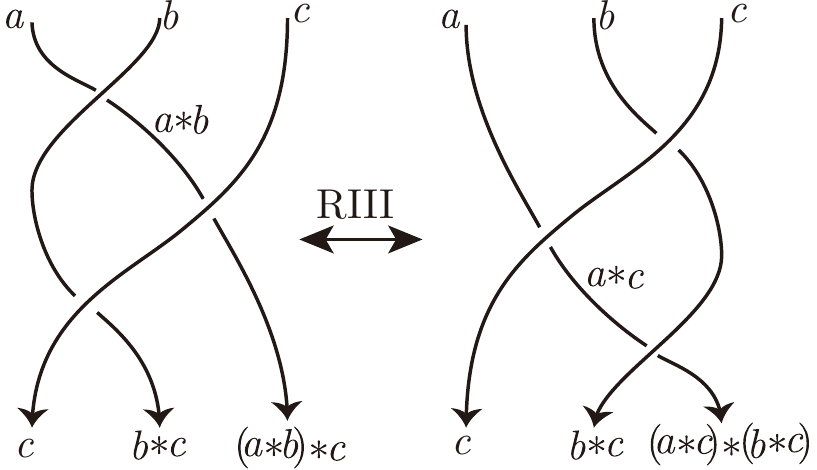} } 
  \caption{Reidemeister moves and the axioms of quandle}\label{pic04}
\end{figure}

\begin{defi} Let $K_1$ and $K_2$ be oriented knots with diagrams $D_1$ and $D_2$, respectively.
For $j=1,2$, let $\rho_j:\pi_1(K_j)\rightarrow{\rm PSL}(2,\mathbb{C})$ be a boundary-parabolic representation. 
For the arc-colorings of $D_1$ and $D_2$ (induced by $\rho_1$ and $\rho_2$, respectively), we make
one arc-color of $D_1$ and another arc-color of $D_2$ coincided by conjugation. 
We denote the coincided arc-color {by} $a\in\mathcal{P}$. 
Then we define the arc-coloring of $D_1\# D_2$ following Figure \ref{pic05}. 
The boundary-parabolic representation induced by this arc-coloring is denoted by 
$$\rho_1\#\rho_2:\pi_1(K_1\# K_2)\rightarrow {\rm PSL}(2,\mathbb{C})$$ and 
is called \textbf{the connected sum of $\rho_1$ and $\rho_2$}.
\end{defi}

\begin{figure}[h]
\centering
\includegraphics[scale=1]{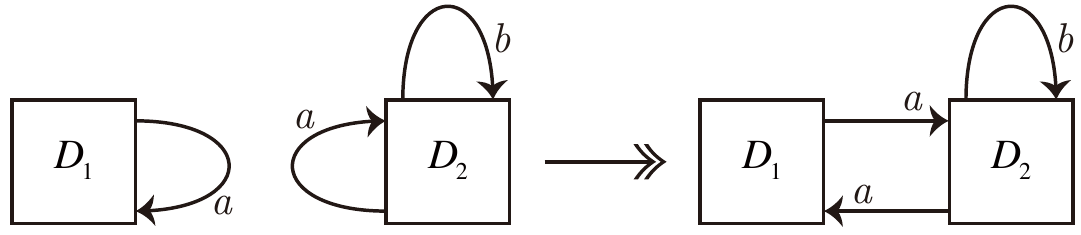}
\caption{Arc-coloring of $D_1\#D_2$}\label{pic05}
\end{figure}

\begin{thm} The connected sum $\rho_1\#\rho_2$ is well-defined up to conjugate.
\end{thm}

\begin{proof} 
At first, note that the well-definedness of $K_1\#K_2$ (up to isotopy) is already proved in standard textbooks.

Let $a\in\mathcal{P}$ be the coincided arc-color in the definition.
For another arc-color $b\in\mathcal{P}$ of $D_2$, there exists unique $c\in\mathcal{P}$ such that $b*c=a$.
We will show the arc-coloring of the right-hand side of Figure \ref{pic06} is conjugate with that of Figure \ref{pic05}.
(In Figure \ref{pic06}, $D_2*c$ means the arc-coloring of $D_2$ obtained by acting $*c$ to all arc-colors.)

\begin{figure}[h]
\centering
\includegraphics[scale=1]{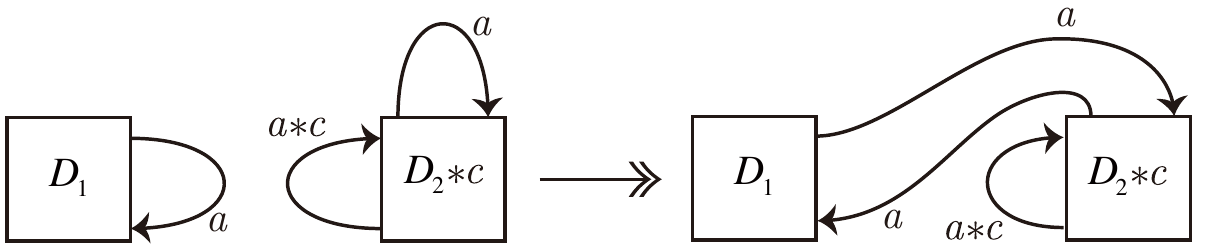}
\caption{Arc-coloring of $D_1\#D_2$ obtained by connecting different arcs}\label{pic06}
\end{figure}

To show the coincidence, we need the observation on the changes of arc-colors in Figure \ref{pic07}.
{ The observation shows that the arc-colors outside $D$ or $D*x$ does not change by moving $D$ or $D*x$ across the crossing. 
Also note that the arc-colors of the two open arcs of $D$ or $D*x$ are always the same.}

\begin{figure}[h]
\centering
\subfigure[Moving under the crossing]{
  \includegraphics[scale=1]{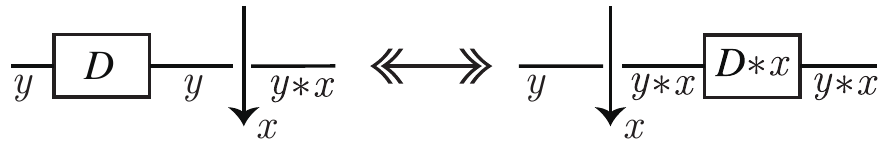}}
\subfigure[Moving over the crossing]{
  \includegraphics[scale=1]{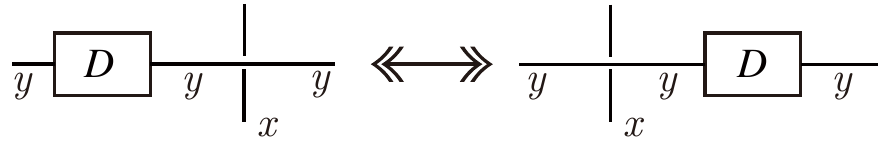} }
  \caption{Changes of arc-colors}\label{pic07}
\end{figure}

{Moving the diagram $D_1$ of the right-hand side of Figure \ref{pic06} (or the left-hand side of Figure \ref{pic08}) inside $D_2*c$,
we obtain the middle picture of Figure \ref{pic08}. (The changed arc-color of $D_1$ is determined by the arc-color $a*c$ of the two arcs.)
By acting $*^{-1}c$ to all arc-colors, we obtain the right-hand side of Figure \ref{pic05}, and the coincidence of the arc-colors is proved.}

\begin{figure}[h]
\centering
\includegraphics[scale=1]{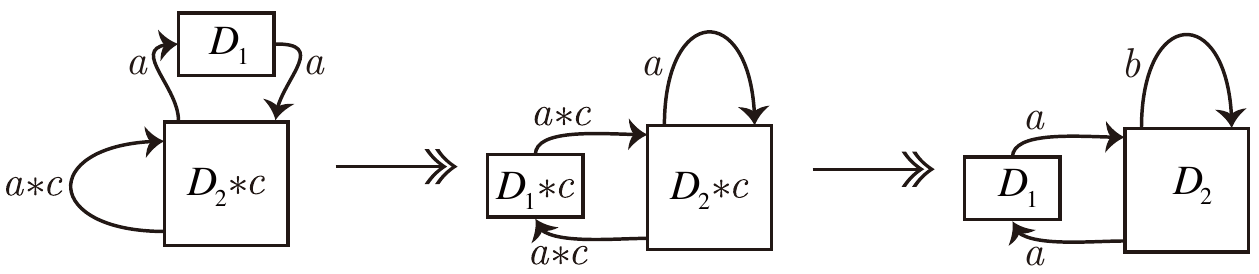}
  \caption{Coincidence of the arc-coloring}\label{pic08}
\end{figure}

On the other hand, the arc-colorings changed by applying Reidemeister moves to the diagrams $D_1$ and $D_2$
are uniquely determined. (See Figure \ref{pic04}.) Therefore, changing diagrams does not have any impact on
the definition of $\rho_1\#\rho_2$.

\end{proof}

\begin{pro}\label{pro1} For a boundary-parabolic representation $\rho:\pi_1(K_1\#K_2)\rightarrow{\rm PSL}(2,\mathbb{C})$,
there exist unique $\rho_1:\pi_1(K_1)\rightarrow{\rm PSL}(2,\mathbb{C})$ and $\rho_2:\pi_1(K_2)\rightarrow{\rm PSL}(2,\mathbb{C})$
satisfying $\rho=\rho_1\#\rho_2$ up to conjugate. (If $K_1=K_2$, then the decomposition is not unique but
$\rho_1\#\rho_2=\rho_2\#\rho_1=\rho$ up to conjugate.)
\end{pro}

\begin{proof}
Choose a diagram $D_1\#D_2$ of $K_1\#K_2$ as in Figure \ref{pic09}(a). Then
the arc-colors $a, b\in\mathcal{P}$ should satisfy $a=b$ because the corresponding meridian loops
are homotopic. Hence we can define $\rho_1$ and $\rho_2$ using the arc-colorings in Figure \ref{pic09}(b).

\begin{figure}[h]
\centering
\subfigure[]{
  \includegraphics[scale=1]{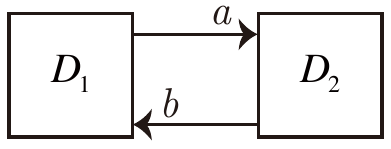}}\hspace{1.5cm}
\subfigure[]{
  \includegraphics[scale=1]{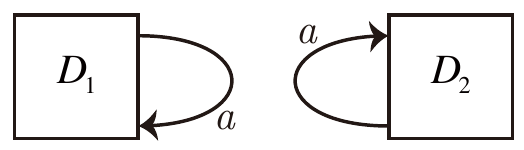} }
  \caption{Arc-colors of $D_1$ and $D_2$ induced by the arc-color of $D_1\#D_2$}\label{pic09}
\end{figure}

To show the uniqueness, assume $\rho_1'\#\rho_2'=\rho=\rho_1\#\rho_2$ up to conjugate.
Then $\rho_1'\#\rho_2'$ also induces an arc-coloring of $D_1\#D_2$, which should be conjugate with
the arc-coloring induced by $\rho$. Therefore, $\rho_1'=\rho_1$ and $\rho_2'=\rho_2$ up to conjugate.

\end{proof}

The general case of $K=K_1\#\ldots\#K_g$ in Section \ref{sec1} can be proved by Proposition \ref{pro1} and the induction on $g$.

Remark that all discussions in this section can be easily generalized to any representation $\rho_j:\pi_1(K_j)\rightarrow {\rm GL}(k,\mathbb{C})$. One obstruction is that, for $\rho_1$ and $\rho_2$,
$\rho_1\#\rho_2$ is defined only when $\rho_1(\alpha)$ is conjugate with $\rho_2(\beta)$ for some meridian loops $\alpha\in\pi_1(K_1)$ and $\beta\in\pi_1(K_2)$.
Also, generalization to links is possible if we specify which components are connected by the connected sum.

\section{Complex volume of $\rho$}\label{sec3}

To calculate the complex volume of $\rho_1\#\rho_2$ explicitly, we briefly review the shadow-coloring of \cite{Cho14a}
and the main result of \cite{Cho14c}.

We identify $\mathbb{C}^2\backslash\{0\}/\pm$ with $\mathcal{P}$ by
\begin{equation}\label{lift}
  \left(\begin{array}{c}\alpha \\\beta\end{array}\right)  
  \longleftrightarrow\left(\begin{array}{cc}1+\alpha\beta & -\alpha^2 \\ \beta^2& 1-\alpha\beta\end{array}\right).
\end{equation}
Then the operation $*$ defined in (\ref{operation}) is given by
\begin{eqnarray*}
  \left(\begin{array}{c}\alpha \\\beta\end{array}\right)*\left(\begin{array}{c}\gamma \\ \delta\end{array}\right)
  =\left(\begin{array}{cc}1+\gamma\delta & -\gamma^2 \\ \delta^2& 1-\gamma\delta\end{array}\right)
  \left(\begin{array}{c}\alpha \\\beta\end{array}\right)\in\mathcal{P},
\end{eqnarray*}
where the operation on the right-hand side is the usual matrix multiplication.
The inverse operation $*^{-1}$ is given by
$$\left(\begin{array}{c}\alpha \\\beta\end{array}\right)*^{-1}\left(\begin{array}{c}\gamma \\ \delta\end{array}\right)
  =\left(\begin{array}{cc}1-\gamma\delta & \gamma^2 \\ -\delta^2& 1+\gamma\delta\end{array}\right)
  \left(\begin{array}{c}\alpha \\\beta\end{array}\right)\in\mathcal{P}.$$

{\it The Hopf map} $h:\mathcal{P}\rightarrow\mathbb{CP}^1=\mathbb{C}\cup\{\infty\}$ is defined by
$$\left(\begin{array}{c}\alpha \\\beta\end{array}\right)\mapsto \frac{\alpha}{\beta}.$$

For the given arc-coloring of the diagram $D$ with arc-colors $a_1,\ldots,a_n$, we assign {\it region-colors} $s_1,\ldots,s_m\in\mathcal{P}$ to regions of $D$ satisfying the rule in Figure \ref{pic10}.
Note that, if an arc-coloring is fixed, then a choice of one region-color determines all the other region-colors.

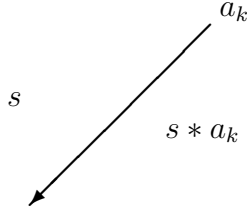
\begin{figure}[h]
\centering  \setlength{\unitlength}{0.6cm}\thicklines
\begin{picture}(6,5)  
    \put(6,4){\vector(-1,-1){4}}
    \put(1.5,2.2){$s$}
    \put(5,1.5){$s*a_k$}
    \put(6.2,4.2){$a_k$}
  \end{picture}
  \caption{Region-coloring}\label{pic10}
\end{figure}

\begin{lem}\label{lem} 
Consider the arc-coloring induced by the boundary-parabolic representation $\rho:\pi_1(K)\rightarrow {\rm PSL}(2,\mathbb{C})$.
Then, for any triple $(a_k,s,s*a_k)$ of an arc-color $a_k$ and its surrounding region-colors $s, s*a_k$ as in Figure \ref{pic10},
there exists a region-coloring satisfying
\begin{equation*}
  h(a_k)\neq h(s)\neq h(s*a_k)\neq h(a_k).
\end{equation*}
\end{lem}

\begin{proof}
See Proof of Lemma 2.4 in \cite{Cho14a}.

\end{proof}

The arc-coloring induced by $\rho$ together with the region-coloring satisfying Lemma \ref{lem}
is called {\it the shadow-coloring induced by} $\rho$. 
We choose $p\in\mathcal{P}$ so that 
\begin{equation}\label{p}
h(p)\notin\{h(a_1),\ldots,h(a_n), h(s_1),\ldots,h(s_m)\}.
\end{equation}

From now on, we fix the representatives of shadow-colors in $\mathbb{C}^2\backslash\{0\}$, not in $\mathcal{P}$. Note that this may cause inconsistency of some signs of arc-colors under the operation $*$.
(In other words, for arc-colors $a_j,a_k,a_l\in\mathcal{P}$ with $a_j=a_k*a_l$, we allow $a_j=\pm a_k*a_l\in \mathbb{C}^2\backslash\{0\}$. As discussed in \cite{Cho14a}, this inconsistency does not make any problem.)
For $a=\left(\begin{array}{c}\alpha_1 \\\alpha_2\end{array}\right)$ and $
b=\left(\begin{array}{c}\beta_1 \\\beta_2\end{array}\right)$ in $\mathbb{C}^2\backslash\{0\}$,
we define {\it the determinant} $\det(a,b)$ by
\begin{equation*}
  \det(a,b):=\det\left(\begin{array}{cc}\alpha_1 & \beta_1 \\\alpha_2 & \beta_2\end{array}\right)=\alpha_1 \beta_2-\beta_1 \alpha_2.
\end{equation*}

For the knot diagram $D$,
we assign variables $w_1,\ldots,w_m$ to the regions with region-colors $s_1,\ldots,s_m$, respectively, 
and define a potential function of a crossing $j$ as in Figure \ref{fig01},
{where $\li(z)=-\int_0^z \frac{\log(1-t)}{t}dt$ is the dilogarithm function.}

\begin{figure}[h]
\setlength{\unitlength}{0.4cm}
\subfigure[Positive crossing]{
  \begin{picture}(35,6)\thicklines
    \put(6,5){\vector(-1,-1){4}}
    \put(2,5){\line(1,-1){1.8}}
    \put(4.2,2.8){\vector(1,-1){1.8}}
    \put(3.5,1){$w_a$}
    \put(5.5,3){$w_b$}
    \put(3.5,4.5){$w_c$}
    \put(1.5,3){$w_d$}
    \put(8,3){$\longrightarrow$}
    \put(11,4){$W_j:=-\li(\frac{w_c}{w_b})-\li(\frac{w_c}{w_d})+\li(\frac{w_a w_c}{w_b w_d})+\li(\frac{w_b}{w_a})+\li(\frac{w_d}{w_a})$}
    \put(15,2){$-\frac{\pi^2}{6}+\log\frac{w_b}{w_a}\log\frac{w_d}{w_a}$}
    \put(3.6,2){$j$}
  \end{picture}}\\
\subfigure[Negative crossing]{
  \begin{picture}(35,6)\thicklines
    \put(2,5){\vector(1,-1){4}}
    \put(6,5){\line(-1,-1){1.8}}
    \put(3.8,2.8){\vector(-1,-1){1.8}}
    \put(3.5,1){$w_a$}
    \put(5.5,3){$w_b$}
    \put(3.5,4.5){$w_c$}
    \put(1.5,3){$w_d$}
    \put(8,3){$\longrightarrow$}
    \put(11,4){$W_j:=\li(\frac{w_c}{w_b})+\li(\frac{w_c}{w_d})-\li(\frac{w_a w_c}{w_b w_d})-\li(\frac{w_b}{w_a})-\li(\frac{w_d}{w_a})$}
    \put(15,2){$+\frac{\pi^2}{6}-\log\frac{w_b}{w_a}\log\frac{w_d}{w_a}$}
        \put(3.6,2){$j$}
  \end{picture}}
  \caption{Potential function of the crossing $j$}\label{fig01}
\end{figure}

Then the potential function of $D$ is defined by
$$W(w_1,\ldots,w_m):=\sum_{j\text{ : crossings}}W_j,$$
and we modify it to
\begin{equation*}
W_0(w_1,\ldots,w_m):=W(w_1,\ldots,w_m)-\sum_{k=1}^m \left(w_k\frac{\partial W}{\partial w_k}\right)\log w_k.
\end{equation*}

Also, from the potential function $W(w_1,\ldots,w_m)$, we define a set of equations
\begin{equation*}
\mathcal{I}:=\left\{\left.\exp\left(w_k\frac{\partial W}{\partial w_k}\right)=1\right|k=1,\ldots,m\right\}.
\end{equation*}
Then, from Proposition 1.1 of \cite{Cho13c}, $\mathcal{I}$ becomes the set of hyperbolicity equations of
the five-term triangulation of $\mathbb{S}^3\backslash (K\cup\{\text{two points}\})$. Here, hyperbolicity equations
are the equations that determine the complete hyperbolic structure of the triangulation, which consist of
gluing equations of edges and completeness condition. 
According to Yoshida's construction in Section 4.5 of \cite{Tillmann13},
a solution $\bold w=(w_1,\ldots,w_m)$ of $\mathcal{I}$ determines the boundary-parabolic representation
\begin{equation*}
\rho_{\bold{w}}:\pi_1(\mathbb{S}^3\backslash (K\cup\{\text{two points}\}))=\pi_1(\mathbb{S}^3\backslash K)\longrightarrow{\rm PSL}(2,\mathbb{C}),
\end{equation*} 
up to conjugate.

\begin{thm}[\cite{Cho14c}]\label{thm2}
For any boundary-parabolic representation $\rho:\pi_1(K)\rightarrow{\rm PSL}(2,\mathbb{C})$ and any knot diagram $D$ of $K$,
there exists the solution $\bold w^{(0)}$ of $\mathcal{I}$ satisfying
$\rho_{\bold w^{(0)}}=\rho,$
up to conjugate. Furthermore, 
\begin{equation}\label{W1}
W_0(\bold{w}^{(0)})\equiv i(\vol(\rho)+i\,\cs(\rho))\modulo.
\end{equation}
The value $\vol(\rho)+i\,\cs(\rho)$ is called \textbf{the complex volume of $\rho$.}
\end{thm}

The explicit formula of $\bold w^{(0)}=(w_1^{(0)},\ldots,w_m^{(0)})$ is very simple. For a region of $D$ with region-color $s_k$ {satisfying} Lemma \ref{lem} and region-variable $w_k$, 
the value $w_k^{(0)}$ of the region-variable is defined by
\begin{equation}\label{main}
w_k^{(0)}:=\det(p,s_k).
\end{equation}

\begin{cor}\label{cor1} For a boundary-parabolic representation $\rho_1\#\rho_2\,:\,\pi_1(K_1\#K_2)\rightarrow {\rm PSL}(2,\mathbb{C})$,
we have
\begin{equation}\label{addvol}
  \vol(\rho_1\#\rho_2)+i\,\cs(\rho_1\#\rho_2)\equiv (\vol(\rho_1)+i\,\cs(\rho_1))+(\vol(\rho_2)+i\,\cs(\rho_2))~~({\rm mod}~i\,\pi^2 ).
\end{equation}
\end{cor}

\begin{proof}
For the connected sum $K_1\#K_2$,
consider a diagram $D_1\#D_2$ and its shadow-coloring induced by $\rho_1\#\rho_2$. 
(Remark that the shadow-coloring satisfies Lemma \ref{lem}.)
By rearranging the indices, we assume
$\{s_1,\ldots,s_l,s_{l+1}\}$ and $\{s_l,s_{l+1},\ldots,s_m\}$ are the region-colors of
$D_1$ and $D_2$, respectively, and $s_l$ is the region-color assigned to the unbounded region of $D_1\#D_2$.
(See Figure \ref{pic12}(a).)

\begin{figure}[h]
\centering
\subfigure[$D_1\#D_2$]{
  \includegraphics[scale=1]{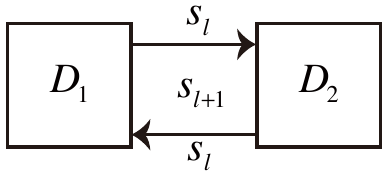}}\hspace{1cm}
  \subfigure[$D_1$]{
  \includegraphics[scale=1]{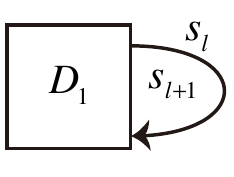} }\hspace{1cm}
  \subfigure[$D_2$]{
  \includegraphics[scale=1]{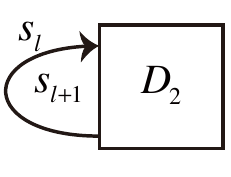} } 
  \caption{Region-colorings of diagrams}\label{pic12}
\end{figure}

Let $W_1(w_1,\ldots,w_l,w_{l+1})$ and $W_2(w_l,w_{l+1},\ldots,w_m)$ be the potential functions
of the diagrams $D_1$ and $D_2$ in Figures \ref{pic12}(b) and (c), respectively.
Then 
$$W(w_1,\ldots,w_m)=W_1(w_1,\ldots,w_l,w_{l+1})+W_2(w_l,w_{l+1},\ldots,w_m)$$ 
holds trivially.

\begin{lem} \label{lem3}
For the solution $\bold w^{(0)}=(w_1^{(0)},\ldots,w_l^{(0)},w_{l+1}^{(0)},\ldots,w_m^{(0)})$ of $\mathcal{I}$ defined by (\ref{main}), let $\bold w_1^{(0)}:=(w_1^{(0)},\ldots,w_l^{(0)},w_{l+1}^{(0)})$ and $\bold w_2^{(0)}:=(w_l^{(0)},w_{l+1}^{(0)},\ldots,w_m^{(0)})$. Then $\bold w_1^{(0)}$ and $\bold w_2^{(0)}$ are solutions of 
$\mathcal{I}_1:=\left\{\left.\exp\left(w_k\frac{\partial W_1}{\partial w_k}\right)=1\right|k=1,\ldots,l,l+1\right\}$
 and $\mathcal{I}_2:=\left\{\left.\exp\left(w_k\frac{\partial W_2}{\partial w_k}\right)=1\right|k=l,l+1,\ldots,m\right\}$, respectively. Furthermore, 
\begin{equation*}
  \rho_{\bold w_j^{(0)}}=\rho_j
\end{equation*}
up to conjugate, and
\begin{equation}\label{volk}
(W_j)_0(\bold w_j^{(0)})\equiv i(\vol(\rho_j)+i\,\cs(\rho_j))\modulo
\end{equation}
for $j=1,2$.\end{lem}

\begin{proof}
Note that the arc-colorings of $D_1$ and $D_2$ induce the representations $\rho_1$ and $\rho_2$, respectively.
Both of the region-colorings $\{s_1,\ldots,s_l,s_{l+1}\}$ and $\{s_l,s_{l+1},\ldots,s_m\}$ of $D_1$ and $D_2$ in
Figures \ref{pic12}(b) and (c), respectively, satisfy Lemma \ref{lem}. 
Therefore, by applying Theorem \ref{thm2} to Figures \ref{pic12}(b) and (c), we obtain the results of this lemma.

\end{proof}

The relation (\ref{addvol}) is directly obtained by (\ref{W1}), (\ref{volk}) and
$$W_0(\bold w^{(0)})=(W_1)_0(\bold w_1^{(0)})+(W_2)_0(\bold w_2^{(0)}),$$
which complete the proof of Corollary \ref{cor1}.

\end{proof}

\section{Twisted Alexander polynomial of $\widetilde{\rho}$}\label{sec4}

To calculate (Wada's) twisted Alexander polynomial, we briefly summarize the calculation method 
in Section 2 of \cite{Morifuji08}.

At first, we lift the boundary-parabolic representation $\rho:\pi_1(K)\rightarrow {\rm PSL}(2,\mathbb{C})$
to $\widetilde{\rho}:\pi_1(K)\rightarrow {\rm SL}(2,\mathbb{C})$
by assuming all arc-colors have trace two. As a matter of fact, this assumption was already reflected
in the right-hand side of (\ref{lift}). Under this lifting, we can trivially obtain
$$\widetilde{\rho_1\#\rho_2}=\widetilde{\rho_1}\#\widetilde{\rho_2}.$$
Therefore, we will use $\widetilde{\rho_1}\#\widetilde{\rho_2}$ instead of $\widetilde{\rho_1\#\rho_2}$
from now on.

Consider the Wirtinger presentation of $\pi_1(K)$ in (\ref{Wpre}). Let 
$$\gamma:\pi_1(K) \rightarrow \mathbb{Z}=<t>$$
be the abelianization homomorphism given by $\gamma(\alpha_1)=\ldots=\gamma(\alpha_n)=t$.
We define the tensor product of $\widetilde{\rho}$ and $\gamma$ by
$$(\widetilde{\rho}\otimes\gamma)(x)=\widetilde{\rho}(x)\gamma(x),$$
for $x\in\pi_1(K)$.

From the maps $\widetilde{\rho}$ and $\gamma$, we obtain natural ring homomorphisms 
$\widetilde{\rho}_*:\mathbb{Z}[\pi_1(K)]\rightarrow M(2,\mathbb{C})$ and
$\gamma_*:\mathbb{Z}[\pi_1(K)]\rightarrow \mathbb{Z}[t,t^{-1}]$,
where $\mathbb{Z}[\pi_1(K)]$ is the group ring of $\pi_1(K)$ and $M(2,\mathbb{C})$ is
the matrix algebra consisting of $2\times 2$ matrices over $\mathbb{C}$.
Combining them, we obtain a ring homomorphism
$$\widetilde{\rho}_*\otimes \gamma_*:\mathbb{Z}[\pi_1(K)]\rightarrow M(2,\mathbb{C}[t,t^{-1}]).$$

Let $F_n=<\alpha_1,\ldots,\alpha_n>$ be the free group and $\psi:\mathbb{Z}[F_n]\rightarrow \mathbb{Z}[\pi_1(K)]$
be the natural surjective homomorphism. Define $\Phi:\mathbb{Z}[F_n]\rightarrow M(2,\mathbb{C}[t,t^{-1}])$ by
$$\Phi=(\widetilde{\rho}_*\otimes\gamma_*)\circ\psi.$$

Consider the $(n-1)\times n$ matrix $M_{\widetilde{\rho}}$ whose $(k,j)$-component is the $2\times 2$ matrix
$$\Phi\left(\frac{\partial r_k}{\partial \alpha_j}\right)\in M(2,\mathbb{C}[t,t^{-1}]),$$
where $\frac{\partial}{\partial \alpha_j}$ denotes the Fox calculus. 
We call $M_{\widetilde{\rho}}$ \textbf{the Alexander matrix associated to $\widetilde\rho$}. We denote by $M_{\widetilde{\rho},\,j}$
the $(n-1)\times (n-1)$ matrix obtained from $M_{\widetilde{\rho}}$ by removing the $j$th column for any $j=1,\ldots,n$. 
Then \textbf{the twisted Alexander polynomial
of $K$ associated to $\widetilde\rho$} is defined by
\begin{equation}\label{def_tA}
\Delta_{K,\,\widetilde{\rho}}\,(t)=\frac{\det M_{\widetilde{\rho},\,j}}{\det \Phi(1-\alpha_j)},
\end{equation}
and it is well-defined up to $t^{p}$ ($p\in\mathbb{Z}$). 

If we concentrate on a boundary-parabolic representation $\rho$ and its lift $\widetilde\rho$, then $\det \Phi(1-\alpha_j)$ in (\ref{def_tA})
is always $(1-t)^2$ independent of the choice of $j$ by the following calculation: after putting 
$\widetilde{\rho}(\alpha_j)=P\left(\begin{array}{cc}1 & 0 \\1 & 1\end{array}\right)P^{-1}$ for certain invertible matrix $P$,
$$\det \Phi(1-\alpha_j)=\det(PP^{-1}-t\,P\left(\begin{array}{cc}1 & 0 \\1 & 1\end{array}\right)P^{-1})
=\det(1-t\left(\begin{array}{cc}1 & 0 \\1& 1\end{array}\right))=(1-t)^2.$$
(Even when we consider a non-boundary-parabolic representation, the value of $\det \Phi(1-\alpha_j)$ in (\ref{def_tA})
is still independent of $j$ because all arc-colors of the knot diagram are conjugate each other.)

Now we apply this calculation method to the case of $K_1\#K_2$ associated to $\widetilde{\rho_1}\#\widetilde{\rho_2}$.
For Figure \ref{pic13}(a), consider the Wirtinger presentation of $\pi_1(K_1)$ and $\pi_1(K_2)$ by
$$\pi_1(K_1)=<\alpha_1,\ldots,\alpha_l \,|\, r_1,\ldots,r_{l-1},r_l>
=<\alpha_1,\ldots,\alpha_l \,|\, r_1,\ldots,r_{l-1}>$$
and
$$\pi_1(K_2)=<\alpha_l,\ldots,\alpha_n \,|\, r_l',r_{l+1},r_{l+2},\ldots,r_n>
=<\alpha_l,\ldots,\alpha_n \,|\, r_{l+1},r_{l+2},\ldots,r_n>,$$
respectively. (In Figure \ref{pic13}, $D_1$ and $D_2$ are the diagrams of $K_1$ and $K_2$, respectively.)

\begin{figure}[h]
\centering
\subfigure[Knot diagrams $D_1$ and $D_2$]{
  \includegraphics[scale=1]{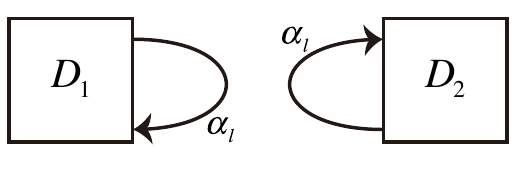}}\hspace{1.5cm}
  \subfigure[Knot diagram $D_1\#D_2$]{
  \includegraphics[scale=1]{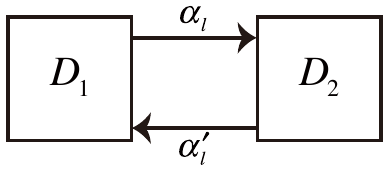} }
  \caption{Knot diagrams with some arcs}\label{pic13}
\end{figure}

\begin{lem}\label{lem_conn} In the above Wirtinger presentation of $\pi_1(K_1)$ and $\pi_1(K_2)$,
we can present $\pi_1(K_1\#K_2)$ by
$$\pi_1(K_1\#K_2)=<\alpha_1,\ldots,\alpha_n\,|\,r_1,\ldots,r_{l-1},r_{l+1},\ldots,r_n>.$$
\end{lem}

\begin{proof}
In Figure  \ref{pic13}(b), the meridian loop corresponding to $\alpha_l$ is homotopic to that of $\alpha_l'$.
Therefore, after writing down the Wirtinger presentation of $\pi_1(K_1\#K_2)$ and substituting $\alpha_l'$ to $\alpha_l$
in all the relations, the resulting presentation is
\begin{equation}\label{rels}
  \pi_1(K_1\#K_2)=<\alpha_1,\ldots,\alpha_n\,|\,r_1,\ldots,r_{l-1},r_{l},r_{l}',r_{l+1},\ldots,r_n>.
\end{equation}
From the fact that $\alpha_l$ is homotopic to $\alpha_l'$, two relations in (\ref{rels}) are redundant, 
one from $D_1$ and another from $D_2$. After removing $r_l$ and $r_l'$, we complete the proof.

\end{proof}

\begin{cor}\label{cor5}
 For the boundary-parabolic representations $\rho_1$, $\rho_2$ and their lifts $\widetilde{\rho_1}$, $\widetilde{\rho_2}$,
the twisted Alexander polynomials satisfy
\begin{equation}\label{cor_pro}
\Delta_{K_1\#K_2,\,\widetilde{\rho_1}\#\widetilde{\rho_2}}
=(1-t)^2 \Delta_{K_1,\,\widetilde{\rho_1}} \Delta_{K_2,\,\widetilde{\rho_2}}.\end{equation}

\end{cor}

\begin{proof} Consider the Wirtinger presentations of $\pi_1(K_1)$, $\pi_1(K_2)$ and $\pi_1(K_1\#K_2)$ above.
Let $M_1$ be the $(l-1)\times(l-1)$ matrix whose $(k,j)$ component is
$$\Phi\left(\frac{\partial r_k}{\partial \alpha_j}\right)~ (k,j=1,\ldots,l-1),$$
and $M_2$ be the $(n-l)\times(n-l)$ matrix whose $(k,j)$ component is
$$\Phi\left(\frac{\partial r_k}{\partial \alpha_j}\right)~ (k,j=l+1,\ldots,n).$$
Then 
\begin{eqnarray*}
\Delta_{K_1\#K_2,\,\widetilde{\rho_1}\#\widetilde{\rho_2}}
&=&\frac{\det\left(\begin{array}{cc}M_1 & 0 \\0 & M_2\end{array}\right)}{\det \Phi(1-\alpha_j)}\\
&=&{\det \Phi(1-\alpha_j)}\frac{\det(M_1)}{\det \Phi(1-\alpha_j)}\frac{\det(M_2)}{\det \Phi(1-\alpha_j)}
=(1-t)^2\Delta_{K_1,\,\widetilde{\rho_1}}\Delta_{K_2,\,\widetilde{\rho_2}}.
\end{eqnarray*}

\end{proof}

{Remark that the natural generalization of the Alexander polynomial $\Delta_K(t)$
is to define the twisted Alexander polynomial $\Delta_{K,\,\widetilde{\rho}}'\,(t)$, using different normalization from (\ref{def_tA}), by
\begin{equation*}
\Delta_{K,\,\widetilde{\rho}}'\,(t):={\det M_{\widetilde{\rho},\,j}}=(1-t)^2\Delta_{K,\,\widetilde{\rho}}\,(t).
\end{equation*}
Then the product formula (\ref{cor_pro}) changes to 
$$\Delta_{K_1\#K_2,\,\widetilde{\rho_1}\#\widetilde{\rho_2}}'
= \Delta_{K_1,\,\widetilde{\rho_1}}' \cdot\Delta_{K_2,\,\widetilde{\rho_2}}',$$
which is a natural generalization of $\Delta_{K_1\# K_2}=\Delta_{K_1}\cdot \Delta_{K_2}$.}

Note that, for non-boundary-parabolic representations of oriented knots, Corollary \ref{cor5} still holds with slight modification.
The term $(1-t)^2$ in (\ref{cor_pro}) should be changed to $\det \Phi(1-\alpha_j)$, 
where $\alpha_j$ is the arc connecting two diagrams.
However, as shown before, choosing any arc $\alpha_k$ instead of the connecting arc $\alpha_j$ gives the same 
equation $\det \Phi(1-\alpha_k)=\det \Phi(1-\alpha_j)$.

\section{Example}\label{sec5}

\begin{figure}[h]
\centering
\includegraphics[scale=1]{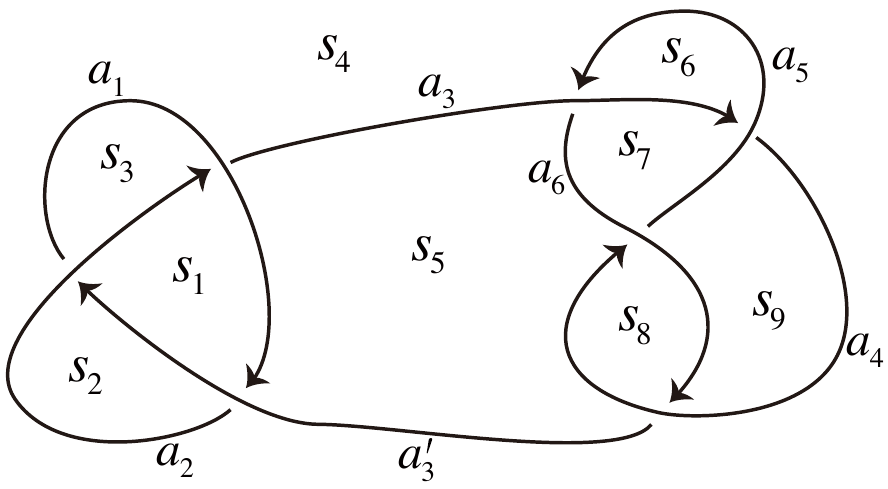}
  \caption{$3_1\#4_1$}\label{pic14}
\end{figure}

For the trefoil knot $3_1$ in the left-hand side and the figure-eight knot $4_1$ in the right-hand side of Figure \ref{pic14},
we put the boundary-parabolic representation $\rho:\pi_1(3_1\#4_1)\rightarrow{\rm PSL}(2,\mathbb{C})$
determined by the arc-colors
\begin{align*}
a_1=\left(\begin{array}{c}-1 \\1\end{array}\right),~a_2=\left(\begin{array}{c}1 \\0\end{array}\right),~
a_3=\left(\begin{array}{c}0 \\1\end{array}\right)=a_3',\\
a_4=\left(\begin{array}{c}x+1 \\x\end{array}\right),~a_5=\left(\begin{array}{c}x \\x\end{array}\right),~
a_6=\left(\begin{array}{c}x \\0\end{array}\right),
\end{align*}
where $x=\frac{-1\pm\sqrt{3}\,i}{2}$ is a solution of $x^2+x+1=0$. 
(We consider each arc-color $a_k$ is assigned to the arc $\alpha_k$.)
Let $\rho=\rho_1\#\rho_2$ for $\rho_j:\pi_1(K_j)\rightarrow{\rm PSL}(2,\mathbb{C})$ with $j=1,2$, and
$\widetilde \rho_1$, $\widetilde \rho_2$, $\widetilde \rho=\widetilde \rho_1\#\widetilde \rho_2$ be their lifts to ${\rm SL}(2,\mathbb{C})$.
If we put $s_1=\left(\begin{array}{c}2 \\1\end{array}\right)$, then all region-colors are uniquely determined by
\begin{align*}
s_1=\left(\begin{array}{c}2 \\1\end{array}\right),~s_2=\left(\begin{array}{c}2 \\3\end{array}\right),~
s_3=\left(\begin{array}{c}1 \\1\end{array}\right),~s_4=\left(\begin{array}{c}-1 \\3\end{array}\right),~
s_5=\left(\begin{array}{c}-1 \\4\end{array}\right),\\
s_6=\left(\begin{array}{c}4x+3 \\4x+7\end{array}\right),~s_7=\left(\begin{array}{c}4x+3 \\4\end{array}\right),~
s_8=\left(\begin{array}{c}4x-2 \\-x-1\end{array}\right),~s_9=\left(\begin{array}{c}3x-2 \\-x-1\end{array}\right).
\end{align*}
Note that this region-coloring satisfies Lemma \ref{lem}. If we put
$$p=\left(\begin{array}{c}1 \\2\end{array}\right),$$
then it satisfies (\ref{p}).

Let $W_1(w_1,\ldots,w_5)$ and $W_2(w_4,\ldots,w_9)$ be the potential functions of $3_1$ and $4_1$ from Figure \ref{pic14}, respectively. 
Then
\begin{align*}
W_1=\left\{-\li(\frac{w_2}{w_1})-\li(\frac{w_2}{w_4})+\li(\frac{w_2 w_3}{w_1 w_4})
+\li(\frac{w_1}{w_3})+\li(\frac{w_4}{w_3})+\log\frac{w_1}{w_3}\log\frac{w_4}{w_3}\right\}\\
+\left\{-\li(\frac{w_3}{w_1})-\li(\frac{w_3}{w_4})+\li(\frac{w_3 w_5}{w_1 w_4})
+\li(\frac{w_1}{w_5})+\li(\frac{w_4}{w_5})+\log\frac{w_1}{w_5}\log\frac{w_4}{w_5}\right\}\\
+\left\{-\li(\frac{w_5}{w_1})-\li(\frac{w_5}{w_4})+\li(\frac{w_2 w_5}{w_1 w_4})
+\li(\frac{w_1}{w_2})+\li(\frac{w_4}{w_2})+\log\frac{w_1}{w_2}\log\frac{w_4}{w_2}\right\}-\frac{\pi^2}{2},
\end{align*} 

\begin{align*}
W_2=\left\{\li(\frac{w_4}{w_5})+\li(\frac{w_4}{w_6})-\li(\frac{w_4 w_7}{w_5 w_6})
-\li(\frac{w_5}{w_7})-\li(\frac{w_6}{w_7})-\log\frac{w_5}{w_7}\log\frac{w_6}{w_7}\right\}\\
+\left\{\li(\frac{w_7}{w_6})+\li(\frac{w_7}{w_9})-\li(\frac{w_4 w_7}{w_6 w_9})
-\li(\frac{w_6}{w_4})-\li(\frac{w_9}{w_4})-\log\frac{w_6}{w_4}\log\frac{w_9}{w_4}\right\}\\
+\left\{-\li(\frac{w_5}{w_7})-\li(\frac{w_5}{w_8})+\li(\frac{w_5 w_9}{w_7 w_8})
+\li(\frac{w_7}{w_9})+\li(\frac{w_8}{w_9})+\log\frac{w_7}{w_9}\log\frac{w_8}{w_9}\right\}\\
+\left\{-\li(\frac{w_9}{w_4})-\li(\frac{w_9}{w_8})+\li(\frac{w_5 w_9}{w_4 w_8})
+\li(\frac{w_4}{w_5})+\li(\frac{w_8}{w_5})+\log\frac{w_4}{w_5}\log\frac{w_8}{w_5}\right\},
\end{align*} 
and the potential function $W(w_1,\ldots,w_9)$ of $3_1\#4_1$ from Figure \ref{pic14} is
$$W(w_1,\ldots,w_9)=W_1(w_1,\ldots,w_5)+W_2(w_4,\ldots,w_9).$$

Let 
\begin{align*}
\mathcal{I}:=\left\{\left.\exp\left(w_k\frac{\partial W}{\partial w_k}\right)=1\right|k=1,\ldots,9\right\},\\
\mathcal{I}_1:=\left\{\left.\exp\left(w_k\frac{\partial W_1}{\partial w_k}\right)=1\right|k=1,\ldots,5\right\},\\
\mathcal{I}_2:=\left\{\left.\exp\left(w_k\frac{\partial W_2}{\partial w_k}\right)=1\right|k=4,\ldots,9\right\},
\end{align*}
and define $\bold w^{(0)}:=(w_1^{(0)},\ldots,w_9^{(0)})$ using the formula (\ref{main}) as follows:
\begin{eqnarray*}
w_1^{(0)}=-3,~w_2^{(0)}=-1,~w_3^{(0)}=-1,~w_4^{(0)}=5,~w_5^{(0)}=6,\\
w_6^{(0)}=-4x+1,~w_7^{(0)}=-8x-2,~w_8^{(0)}=-9x+3,~w_9^{(0)}=-7x+3.
\end{eqnarray*}
We put $\bold w_1^{(0)}=(w_1^{(0)},\ldots,w_5^{(0)})$ and $\bold w_2^{(0)}=(w_4^{(0)},\ldots,w_9^{(0)})$.
Then $\bold w_1^{(0)}$, $\bold w_2^{(0)}$ and $\bold w^{(0)}$ are solutions of 
$\mathcal{I}_1$, $\mathcal{I}_2$ and $\mathcal{I}$, respectively. 
Furthermore, numerical calculation shows
\begin{eqnarray*}
i(\vol(\rho_1)+i\,\cs(\rho_1))\equiv (W_1)_0(\bold w_1^{(0)})&\equiv&i(0+1.6449...i)~~({\rm mod}~\pi^2),\\
i(\vol(\rho_2)+i\,\cs(\rho_2))\equiv (W_2)_0(\bold w_2^{(0)})&\equiv&
\left\{\begin{array}{ll}i(2.0299...+0\,i)&\text{ if }x=\frac{-1-\sqrt{3} \,i}{2} \\
                i(-2.0299...+0\,i)&\text{ if }x=\frac{-1+\sqrt{3}\,i}{2} \end{array}\right.
~~({\rm mod}~\pi^2),\\
\end{eqnarray*}
and
\begin{eqnarray*}
\lefteqn{i(\vol(\rho_1\#\rho_2)+i\,\cs(\rho_1\#\rho_2))\equiv W_0(\bold w^{(0)})}\\
&&\equiv\left\{\begin{array}{ll}i(2.0299...+1.6449...\,i)&\text{ if }x=\frac{-1-\sqrt{3} \,i}{2} \\
                i(-2.0299...+1.6449...\,i)&\text{ if }x=\frac{-1+\sqrt{3}\,i}{2} \end{array}\right.\\
&&\equiv(W_1)_0(\bold w_1^{(0)})+(W_2)_0(\bold w_2^{(0)})
\equiv i(\vol(\rho_1)+i\,\cs(\rho_1))+i(\vol(\rho_2)+i\,\cs(\rho_2))~~({\rm mod}~\pi^2),
\end{eqnarray*}
which confirms the additivity of the complex volume in Corollary \ref{cor1}. 

To calculate the twisted Alexander polynomials, we put the Wirtinger presentations of $3_1$, $4_1$ and $3_1\#4_1$ from
Figure \ref{pic14} by
\begin{eqnarray*}
\pi_1(3_1)&=&<\alpha_1,\alpha_2,\alpha_3\,|\,\alpha_1\alpha_2\alpha_1^{-1}\alpha_3^{-1},
\alpha_2\alpha_3\alpha_2^{-1}\alpha_1^{-1},\alpha_3\alpha_1\alpha_3^{-1}\alpha_2^{-1}>\\
&=&<\alpha_1,\alpha_2,\alpha_3\,|\,\alpha_1\alpha_2\alpha_1^{-1}\alpha_3^{-1},
\alpha_2\alpha_3\alpha_2^{-1}\alpha_1^{-1}>,\\
\pi_1(4_1)&=&<\alpha_3,\alpha_4,\alpha_5,\alpha_6\,|\,
  \alpha_3\alpha_6\alpha_3^{-1}\alpha_5^{-1},\alpha_5\alpha_4\alpha_5^{-1}\alpha_3^{-1},\alpha_6\alpha_4\alpha_6^{-1}\alpha_5^{-1}>,\\
\pi_1(3_1\#4_1)&=&<\alpha_1,\alpha_2,\alpha_3,\alpha_3',\alpha_4,\alpha_5,\alpha_6\,|\,
  \alpha_1\alpha_2\alpha_1^{-1}\alpha_3^{-1},\alpha_2\alpha_3'\alpha_2^{-1}\alpha_1^{-1},\\
  &&~~~\alpha_3'\alpha_1(\alpha_3')^{-1}\alpha_2^{-1},\alpha_3\alpha_6\alpha_3^{-1}\alpha_5^{-1},
  \alpha_5\alpha_4\alpha_5^{-1}\alpha_3^{-1},\alpha_6\alpha_4\alpha_6^{-1}\alpha_5^{-1}>,
\end{eqnarray*}
respectively. 
(If we use Lemma \ref{lem_conn}, the fundamental group $\pi_1(3_1\#4_1)$ can be expressed simply by
\begin{eqnarray*}
\pi_1(3_1\#4_1)&=&<\alpha_1,\alpha_2,\alpha_3,\alpha_4,\alpha_5,\alpha_6\,|\,
  \alpha_1\alpha_2\alpha_1^{-1}\alpha_3^{-1},\alpha_2\alpha_3\alpha_2^{-1}\alpha_1^{-1},\\
  &&~~~\alpha_3\alpha_6\alpha_3^{-1}\alpha_5^{-1},\alpha_5\alpha_4\alpha_5^{-1}\alpha_3^{-1},
  \alpha_6\alpha_4\alpha_6^{-1}\alpha_5^{-1}>.
\end{eqnarray*}
This presentation shows {(\ref{cor_pro})} trivially, so we are using the Wirtinger presentation of $\pi_1(3_1\#4_1)$ instead.)
The Alexander matrices associated to $\widetilde \rho_1$, $\widetilde \rho_2$ and $\widetilde \rho_1\#\widetilde \rho_2$
obtained by the above Wirtinger presentations are
$$
M_{\widetilde \rho_1}=\left(\begin{array}{cccccc}
  1-t & 0 & 0 & -t & -1 & 0 \\
  -t & 1-t & t & 2t & 0 & -1 \\
  -1 & 0 & 1 & t & t & -t \\
  0 & -1 & -t & 1-2t & 0 & t\end{array}\right),$$
$$M_{\widetilde \rho_2}=\left(\begin{array}{cccccccc}
  1+xt & -(x+1)t & 0 & 0 & -1 & 0 & t & 0 \\
  (x+1)t & 1-(x+2)t & 0 & 0 & 0 & -1 & t & t \\
  -1 & 0 & -xt & (x+1)t & 1-t & 0 & 0 & 0 \\
  0 & -1 & -(x+1)t & (x+2)t & -t & 1-t & 0 & 0 \\
  0 & 0 & t & (x+1)t & -1 & 0 & 1+xt & -(x+1)t \\
  0 & 0 & 0 & t & 0 & -1 & (x+1)t & 1-(x+2)t
  \end{array}\right),$$
and $M_{\widetilde \rho_1\#\widetilde \rho_2}=$
$$
\scriptsize\left(\begin{array}{cccccccccccccc}
  1-t & 0 & 0 & -t & -1 & 0 & 0 & 0 & 0 & 0 & 0 & 0 & 0 & 0\\
  -t & 1-t & t & 2t & 0 & -1 & 0 & 0 & 0 & 0 & 0 & 0 & 0 & 0\\
  
  -1 & 0 & 1 & t & 0 & 0 & t & -t & 0 & 0 & 0 & 0 & 0 & 0\\
  0 & -1 & -t & 1-2t & 0 & 0 & 0 & t & 0 & 0 & 0 & 0 & 0 & 0\\
  
  t & 0 & -1 & 0 & 0 & 0 & 1-t & t & 0 & 0 & 0 & 0 & 0 & 0\\
  t & t & 0 & -1 & 0 & 0 & 0 & 1-t & 0 & 0 & 0 & 0 & 0 & 0\\
  
  0 & 0 & 0 & 0 &   1+xt & -(x+1)t &0 & 0 & 0 & 0 & -1 & 0 & t & 0 \\
  0 & 0 & 0 & 0 &   (x+1)t & 1-(x+2)t &0 & 0 & 0 & 0 & 0 & -1 & t & t\\
  
  0 & 0 & 0 & 0 &   -1 & 0 &0 & 0 & -xt & (x+1)t & 1-t & 0 & 0 & 0 \\
  0 & 0 & 0 & 0 &   0 & -1 &0 & 0 & -(x+1)t & (x+2)t & -t & 1-t & 0 & 0 \\
  
  0 & 0 & 0 & 0 &   0 & 0 &0 & 0 & t & (x+1)t & -1 & 0 & 1+xt & -(x+1)t\\
  0 & 0 & 0 & 0 &   0 & 0 &0 & 0 & 0 & t & 0 & -1 & (x+1)t & 1-(x+2)t
\end{array}\right),$$
respectively. The corresponding twisted Alexander polynomials 
obtained by (\ref{def_tA}) are
\begin{eqnarray*}
\Delta_{3_1,\,\widetilde\rho_1}(t)&=&1+t^2,\\
\Delta_{4_1,\,\widetilde\rho_2}(t)&=&t^2(1-4t+t^2),\\
\Delta_{3_1\#4_1,\,\widetilde \rho_1\#\widetilde \rho_2}(t)&=&(1-t)^2(1+t^2)t^2(1-4t+t^2),
\end{eqnarray*}
respectively.\footnote{To calculate the determinants, final two columns of all three matrices are removed.} Therefore, we obtain
\begin{equation}\label{final}\Delta_{3_1\#4_1,\,\widetilde \rho_1\#\widetilde \rho_2}(t)
=(1-t)^2\Delta_{3_1,\,\widetilde\rho_1}(t)\Delta_{4_1,\,\widetilde\rho_2}(t),\end{equation}
which confirms Corollary \ref{cor5}.

\appendix
\section{Errata}\label{errata}

This appendix is the errata of this article. (The author appreciates Seonhwa Kim for pointing out the error.)
The author found the errors after the publication, so he wrote this errata and submitted it to the same journal again.
He sincerely apologizes to the readers for confusing them.

For given boundary-parabolic representations $\rho_j : \pi_1(K_j)\rightarrow {\rm PSL}(2,\mathbb{C})$ ($j=1,2$), the connected sum
$\rho_1\#\rho_2 : \pi_1(K_1\#K_2)\rightarrow {\rm PSL}(2,\mathbb{C})$ was defined at this article. He proved that $\rho_1\#\rho_2$ is well-defined up to conjugation at Theorem 2.2, but the statement and the proof are not correct.

As an counterexample of Theorem 2.2, consider the example of Fig 14 in Section 5.
We put
\begin{align}
a_1=\left(\begin{array}{c}-1 \\1\end{array}\right),~a_2=\left(\begin{array}{c}1 \\0\end{array}\right),~
a_3=\left(\begin{array}{c}0 \\1\end{array}\right)=a_3',\label{eq1}\\
a_4=\left(\begin{array}{c}x+1 \\x\end{array}\right),~a_5=\left(\begin{array}{c}x \\x\end{array}\right),~
a_6=\left(\begin{array}{c}x \\0\end{array}\right),\nonumber
\end{align}
but we can conjugate the figure-eight knot part by the map $*a_3:\mathcal{P}\rightarrow\mathcal{P}$.
The changed arc-colors are
\begin{align}
a_1=\left(\begin{array}{c}-1 \\1\end{array}\right),~a_2=\left(\begin{array}{c}1 \\0\end{array}\right),~
a_3=\left(\begin{array}{c}0 \\1\end{array}\right)=\left(\begin{array}{c}0 \\1\end{array}\right)*\left(\begin{array}{c}0 \\1\end{array}\right)=a_3',\label{eq2}\\
a_4=\left(\begin{array}{c}x+1 \\x\end{array}\right)*\left(\begin{array}{c}0 \\1\end{array}\right)=\left(\begin{array}{c}x+1 \\2x+1\end{array}\right)
,~a_5=\left(\begin{array}{c}x \\x\end{array}\right)*\left(\begin{array}{c}0 \\1\end{array}\right)=\left(\begin{array}{c}x \\2x\end{array}\right)
,\nonumber\\
a_6=\left(\begin{array}{c}x \\0\end{array}\right)*\left(\begin{array}{c}0 \\1\end{array}\right)=\left(\begin{array}{c}x \\x\end{array}\right)
.\nonumber
\end{align}
The representations defined by (\ref{eq1}) and (\ref{eq2}) cannot be conjugate, so the connected sum cannot be well-defined.

The error lies in the third sentence of the proof of Theorem 2.2:
``For any $b\in\mathcal{P}$, there exists unique $c\in\mathcal{P}$ such that $b*c=a$.''
The author confused that the bijectiveness of the map $*c$ implies this statement.
This statement is wrong, so all of the proof is wrong. 
(For example, in Fig 8, the second diagram is wrong. We cannot guarantee the arc-color of the small box becomes $D_1*c$.)
Therefore, Theorem 2.2 is wrong and Definition 2.1 should be modified.

One way to solve these errors is to consider the connected sum $\rho_1\#\rho_2$ not as a {\it definition}, but a method
to construct boundary-parabolic representations.
This construction does not define the unique representation, but it defines many representations.
Under this construction, Proposition 2.3 should be changed as follows.

\begin{pro}[New version of Proposition 2.3] For a boundary-parabolic representation $\rho:\pi_1(K_1\#K_2)\rightarrow {\rm PSL}(2,\mathbb{C})$, 
there exist $\rho_1:\pi_1(K_1)\rightarrow {\rm PSL}(2,\mathbb{C})$ and $\rho_2:\pi_1(K_2)\rightarrow {\rm PSL}(2,\mathbb{C})$, 
which are unique up to conjugation, such that one of $\rho_1\#\rho_2$ becomes $\rho$.
\end{pro}

\begin{proof} The existence is trivial from Fig 9. The uniqueness follows from Fig 7 because
the arc-color of $D$ is invariant under the moves up to conjugation.
\end{proof}

Interestingly, Section 3--4 are still true under this construction. This implies that
any representation obtained by $\rho_1\#\rho_2$ has the same complex volume 
$$(\vol(\rho_1)+i\,\cs(\rho_1))+(\vol(\rho_2)+i\,\cs(\rho_2))$$
 and the same twisted Alexander polynomial
$$
\Delta_{K_1,\,\widetilde{\rho_1}}' \cdot\Delta_{K_2,\,\widetilde{\rho_2}}'.$$

\vspace{5mm}
\begin{ack}
The author appreciates Teruaki Kitano for giving very nice introductory lectures on twisted Alexander polynomial
at Seoul National University in November, 2014. Section \ref{sec4} of this article is motived by his talk.
Also, discussions with Sungwoon Kim, Yuichi Kabaya and Hyuk Kim helped the author a lot for preparing this article.
\end{ack}

{
\begin{flushleft}
  Pohang Mathematics Institute (PMI),\\ Pohang 37673, Republic of Korea\\
  \vspace{0.4cm}
E-mail: dol0425@gmail.com\\
\end{flushleft}}


\begin{thebibliography}{1}

\bibitem{Cho13c}
J.~Cho.
\newblock Optimistic limits of colored {J}ones polynomials and complex volumes
  of hyperbolic links.
\newblock arXiv:1303.3701, 03 2013.

\bibitem{Cho14c}
J.~Cho.
\newblock Optimistic limit of the colored {J}ones polynomial and the existence
  of a solution.
\newblock arXiv:1410.0525, 10 2014.

\bibitem{Cho14a}
J.~Cho.
\newblock Quandle theory and optimistic limits of representations of knot
  groups.
\newblock arXiv:1409.1764, 09 2014.

\bibitem{Kabaya14}
A.~Inoue and Y.~Kabaya.
\newblock Quandle homology and complex volume.
\newblock {\em Geom. Dedicata}, 171:265--292, 2014.

\bibitem{Tillmann13}
F.~Luo, S.~Tillmann, and T.~Yang.
\newblock Thurston's spinning construction and solutions to the hyperbolic
  gluing equations for closed hyperbolic 3-manifolds.
\newblock {\em Proc. Amer. Math. Soc.}, 141(1):335--350, 2013.

\bibitem{Morifuji08}
T.~Morifuji.
\newblock Twisted {A}lexander polynomials of twist knots for nonabelian
  representations.
\newblock {\em Bull. Sci. Math.}, 132(5):439--453, 2008.

\bibitem{Zickert09}
C.~K. Zickert.
\newblock The volume and {C}hern-{S}imons invariant of a representation.
\newblock {\em Duke Math. J.}, 150(3):489--532, 2009.

\end{thebibliography}
\end{document}